\documentclass{amsart}
\usepackage[utf8]{inputenc}
\usepackage{geometry,amsmath,amsthm, amssymb,graphicx,tikz,tikz-cd,mathrsfs,fancyhdr,wrapfig,enumitem,mathtools,verbatim}
\usepackage{hyperref, float}
\hypersetup{colorlinks=true,citecolor=purple,linkcolor=purple}
\numberwithin{equation}{section}

\usepackage[backrefs]{amsrefs}

\theoremstyle{definition}

\newtheorem{thm*}{Theorem}
\newtheorem{theorem}{Theorem}[section]

\newtheorem{definition}[theorem]{Definition}

\newtheorem{conj*}{Conjecture}

\newtheorem{prop}[theorem]{Proposition}

\newtheorem{lemma}[theorem]{Lemma}

\newtheorem{Counterex}[theorem]{Counterexample}

\newcommand{\exref}[1]{\hyperref[#1]{Example \ref{#1}}}
\newcommand{\figref}[1]{\hyperref[#1]{Figure \ref{#1}}}
\newcommand{\lemref}[1]{\hyperref[#1]{Lemma \ref{#1}}}
\newcommand{\thmref}[1]{\hyperref[#1]{Theorem \ref{#1}}}
\newcommand{\thmsref}[1]{\hyperref[#1]{Theorems \ref{#1}}}
\newcommand{\probref}[1]{\hyperref[#1]{Problem \ref{#1}}}
\newcommand{\conjref}[1]{\hyperref[#1]{Conjecture \ref{#1}}}
\newcommand{\propref}[1]{\hyperref[#1]{Proposition \ref{#1}}}
\newcommand{\corref}[1]{\hyperref[#1]{Corollary \ref{#1}}}
\newcommand{\defref}[1]{\hyperref[#1]{Definition \ref{#1}}}
\newcommand{\rmkref}[1]{\hyperref[#1]{Remark \ref{#1}}}
\newcommand{\rmksref}[1]{\hyperref[#1]{Remarks \ref{#1}}}
\newcommand{\qref}[1]{\hyperref[#1]{Question \ref{#1}}}
\newcommand{\secref}[1]{\hyperref[#1]{\S\ref{#1}}}
\newcommand{\appref}[1]{\hyperref[#1]{Appendix \ref{#1}}}
\newcommand{\tabref}[1]{\hyperref[#1]{Table \ref{#1}}}
\newcommand{\claimref}[1]{\hyperref[#1]{Claim \ref{#1}}}

\newcommand\pf{\begin{proof}}
\newcommand\epf{\end{proof}}

\newcommand\bp{\begin{pmatrix}}
\newcommand\ep{\end{pmatrix}}
\newcommand\ben{\begin{enumerate}}
\newcommand\een{\end{enumerate}}
\newcommand\be{\begin{equation}}
\newcommand\ee{\end{equation}}
\newcommand\benn{\begin{equation*}}
\newcommand\eenn{\end{equation*}}
\newcommand\bea{\begin{eqnarray}}
\newcommand\eea{\end{eqnarray}}
\newcommand\beann{\begin{eqnarray*}}
\newcommand\eeann{\end{eqnarray*}}

\newcommand\irr{\text{irr}}

\newcommand{\bP}{\mathbb{P}}

\newcommand{\Z}{\mathbb{Z}}

\newcommand{\cF}{\mathcal{F}}

\newcommand{\ind}{\operatorname{ind}}

\DeclareMathOperator{\ann}{ann}

\newcommand{\abs}[1]{\left\lvert#1\right\rvert}

\DeclareMathOperator{\codim}{codim}

\renewcommand{\ind}{\text{ind}}

\newcommand*{\Cdot}[1][1.25]{%
  \mathpalette{\CdotAux{#1}}\cdot%
}
\newdimen\CdotAxis
\newcommand*{\CdotAux}[3]{%
  {%
    \settoheight\CdotAxis{$#2\vcenter{}$}%
    \sbox0{%
      \raisebox\CdotAxis{%
        \scalebox{#1}{%
          \raisebox{-\CdotAxis}{%
            $\mathsurround=0pt #2#3$%
          }%
        }%
      }%
    }%
    \dp0=0pt %
    \sbox2{$#2\bullet$}%
    \ifdim\ht2<\ht0 %
      \ht0=\ht2 %
    \fi
    \sbox2{$\mathsurround=0pt #2#3$}%
    \hbox to \wd2{\hss\usebox{0}\hss}%
  }%
}

\newcommand{\CDot}{\Cdot[2]}

\title{On Virtually Cohen--Macaulay Simplicial Complexes}
\author{Nathan Kenshur, Feiyang Lin, Sean McNally, Zixuan Xu, Teresa Yu}
\date{\today}

\begin{document}

\maketitle

\begin{abstract}
    We examine virtual resolutions of Stanley--Reisner ideals for a product of projective spaces. In particular, we provide sufficient conditions for a simplicial complex to be virtually Cohen--Macaulay (to have a virtual resolution with length equal to its codimension). We also show that all balanced simplicial complexes are virtually Cohen--Macaulay.
\end{abstract}

\section{Introduction}
In \cite{BES}, Berkesch, Erman, and Smith introduced \textit{virtual resolutions} of graded modules to take into account the freedom given by nonmaximal irrelevant ideals in the Cox ring $S=k[\underline{x}]$ of most smooth projective toric varieties, such as a product of projective spaces. Meanwhile, there is a nice class of modules with the \textit{Cohen--Macaulay} property, which is characterized by having a minimal free resolution of length equal to their codimenion. While a significant amount of research has gone into characterizing Cohen--Macaulay modules, there is still little understanding of when modules are \textit{virtually Cohen--Macaulay}, meaning that they have virtual resolutions of length equal to their codimension. 

For $\vec{n} = (n_1,\ldots, n_r)$, let $\bP^{\vec n} = \bP^{n_1} \times \cdots \times \bP^{n_r}$ be a product of projective spaces. Let $S\vcentcolon=\Bbbk[x_{i,j}\mid 1 \leq i \leq r,\,0 \leq j \leq n_i]$ be the Cox ring of $\bP^{\vec n}$ over a field $\Bbbk$. Let $B \vcentcolon= \bigcap_{i = 1}^r \langle x_{i,0}, x_{i,1}, \ldots, x_{i,n_i} \rangle$ be the \textit{irrelevant ideal of $S$}. We endow $S$ with a $\Z^r$-grading given by $\deg(x_{i,j})= \bf{e}_i$, the $i^{\text{th}}$ standard basis vector. For convenience, we write $|\vec n| := n_1 + \cdots + n_r$. 

A simplicial complex $\Delta$ on $\bP^{\vec{n}}$ has vertices in $X_{\vec{n}} = \{x_{i,j}\mid 1\leq i\leq r,\,0\leq j\leq n_i\}$. When there is no potential confusion, we use $X$ as a shorthand for $X_{\vec{n}}$. We use \textit{component} to refer to a term $\mathbb{P}^{n_i}$ of the product of projective spaces. We also use the term component in the context of simplicial complexes to refer to the set of vertices of a complex associated to a component $\mathbb{P}^{n_i}$. Given a simplicial complex $\Delta$ on vertex set $X$ and a squarefree monomial ideal $I \subset S$, let $I_\Delta \subseteq \Bbbk[X]$ and $\Delta_I$ be, respectively, the squarefree monomial ideal and simplicial complex given by the Stanley--Reisner correspondence.

The saturation of $I$ by $B$ is the ideal $I : B^\infty := \bigcup_{k > 0} (I: B^k)$. A simpicial complex $\Delta$ is \textit{$B$-saturated} if $I_\Delta:B^\infty = I_\Delta$. 
Define $\Gamma_B(M) := \{r \in M \mid rB^k = 0\}.$ 

\begin{definition}\label{def::vres}
Let $I$ be an ideal in $S$. A graded free complex $\cF_{\CDot}\colon [F_0 \leftarrow F_1\leftarrow \cdots \leftarrow F_k \leftarrow 0]$ is a \textit{virtual resolution of $S/I$} if:
\begin{enumerate}
    \item $\sqrt{\ann \left(H_i \cF_{\CDot}\right)} \supseteq B$ for all $i>0$, and
    \item $\ann \left(H_0 \cF_{\CDot}/\Gamma_B(H_0 \cF_{\CDot})\right) = (S/I)/\Gamma_B(S/I)$.
\end{enumerate}
Alternatively, $\cF_{\CDot}$ is a virtual resolution of $S/I$ if $\widetilde{\cF_{\CDot}}$ is a locally free resolution of the sheaf $(S/I)^\sim$. 
In the special case where $F_0=S^1$, an equivalent formulation of the second condition is
\[\ann \left(H_0 \cF_{\CDot} : B^{\infty} \right)= I : B^{\infty}.\]
\end{definition}

\begin{definition}\label{defn:VCM}
Let $\codim I_\Delta := |\vec{n}| - \dim V(I_\Delta)$, where $V(I_\Delta)$ is the variety of $I_\Delta$ in $\bP^{\vec{n}}$. A $B$-saturated simplicial complex $\Delta$ in a product projective spaces $\bP^{\vec{n}}$ supported on  vertex set $X$ is \textit{virtually Cohen--Macaulay} if there exists a virtual resolution of $\Bbbk[\Delta]$ of length $\codim \left(I_{\Delta}\right)$.
\end{definition}

In this paper, we restrict our attention to a specific class of $S$-modules of the form $S/I_{\Delta}$, where $I_\Delta$ is a squarefree monomial ideal defined by a simplicial complex $\Delta$ and $S$ is an appropriate polynomial ring. A simplicial complex $\Delta$ is Cohen--Macaulay if $(S/I_\Delta)_{\langle \underline{x} \rangle}$ is a Cohen--Macaulay ring. A Cohen--Macaulay simplicial complex has many nice combinatorial properties and characterizations, such as Reisner's criterion and connections to purity, shellability, and gallery-connectedness. Here, we consider whether similar statements can be made about virtually Cohen--Macaulay simplicial complexes and show that virtually Cohen--Macaulay simplicial complexes are not necessarily gallery-connected. We also show in Counterexample~\ref{counterex} that one cannot obtain a Cohen--Macaulay complex from a virtually Cohen--Macaulay complex solely by manipulating irrelevant faces.

Let $\Delta$ be a pure simplicial complex on the product of projective spaces $\bP^{\vec{n}}$. A facet $F\in \Delta$ is \textit{balanced} if it contains exactly one vertex from every component. A simplicial complex is \textit{balanced} if all of its facets are balanced. We consider when balanced simplicial complexes are virtually Cohen--Macaulay and obtain the following main theorem.

\begin{theorem}\label{thm:balanced_implies_vcm}
If $\Delta$ is a pure and balanced simplicial complex on vertex set $V$ corresponding to the product of projective spaces $\bP^{\vec{n}}$ , then $\Delta$ is virtually Cohen--Macaulay.
\end{theorem}

\subsection*{Outline.}
The rest of the paper is organized as follows.
In \secref{sec:background}, we provide necessary background on Stanley--Reisner theory and virtual resolutions. We also show in \secref{sec:background} that virtually Cohen--Macaulay complexes are pure after removing irrelevant facets. In \secref{sec:partial_char_vcm}, we show that if a short virtual resolution is found for $S/I_\Delta$ by taking the free resolution of some module, then there exists a way of adding irrelevant faces to the resulting simplicial complex $\Delta$ so that it has a short free resolution. We give a partial characterization (Theorem~\ref{thm:radicalVCM}) of virtually Cohen--Macaulay complexes for which we can obtain a short virtual resolution by intersecting $I_\Delta$ with ideals with empty variety, such as the irrelevant ideal. 
Finally, we restrict ourselves to the special case of balanced complexes in \secref{sec:balanced} and show that every balanced complex is in fact virtually Cohen--Macaulay.

\section{Background}\label{sec:background}

In this section, we provide some preliminary results on Stanley--Reisner theory and virtual resolutions.

Notice that when $I_{\Delta}$ is $B$-saturated, we have that $n-\dim V(I_{\Delta})=n+r-\dim_{\mathbb{A}^{n+r}}I_{\Delta}$, which means that the geometric definition of codimension  coincides with the definition in terms of the Krull dimension of $I_{\Delta}$ in the Cox ring. Since a minimal free resolution of a module is also a virtual resolution of the module, all Cohen--Macaulay simplicial complexes are virtually Cohen--Macaulay. The following lemma provides an important method for obtaining virtual resolutions of some $S/I$.

\begin{lemma}\label{lem::varvres}
If $I,J \subseteq S$ are ideals such that $V(I)=V(J)$, then any free resolution $\cF_{\CDot}$ of $S/J$ is a virtual resolution of $S/I$. \qed
\end{lemma}

Lemma~\ref{lem::varvres} tells us that in the virtual context, one is ``free'' to modify an ideal up to having the same variety. By the Stanley--Reisner correspondence, this leads to the notion of irrelevant and relevant faces in a simplicial complex. A face is irrelevant if its presence or absence does not affect the variety of the ideal correponding to the simplicial complex. This is equivalent to the following definition.

We say that a face of a simplicial complex $\Delta$ is \textit{relevant} if it contains at least one vertex from every component of $\mathbb{P}^{\vec{n}}$; otherwise, we say it is \textit{irrelevant}.

The following lemma provides a useful method for obtaining virtually Cohen--Macaulay simplicial complexes.

\begin{lemma}\label{lem:virtualequivalence}
Let $\Delta, \Delta'$ be two simplicial complexes in $\bP^{\vec{n}}$ such that $\Delta\subseteq\Delta'$ and $\Delta'\setminus \Delta$ contains only irrelevant faces. Then the free resolution of $S/I_{\Delta'}$ is a virtual resolution of $S/I_{\Delta}$.
\end{lemma}

For $\Delta$ supported on $X$ and $A \subseteq X$, by definition of Stanley--Reisner ideals, we have that $Q = \langle x_i\mid x_i \in A\rangle$ is an associated prime of $I_\Delta$ if and only if $X \setminus A$ is a facet of $\Delta$. Moreover, $\codim (Q) = |\vec{n}| + r - 1 - \dim (X \setminus A)$. Thus, $\Delta$ being pure up to irrelevant facets is equivalent to the following proposition.

\begin{prop}
Let $\Delta$ be a simplicial complex on the product projective space $\mathbb{P}^{\vec n} = \mathbb{P}^{n_1}\times \cdots \times \mathbb{P}^{n_r}$, where $\vec n = (n_1, \dots, n_r)$. Let $F := [F_0 \leftarrow F_1 \cdots \leftarrow F_p \leftarrow 0]$ be a virtual resolution of $S/I_{\Delta}$ such that $p = |\vec n| - \dim V(I_\Delta)$. Then if $Q$ is an associated prime of $I_\Delta$ that does not contain the irrelevant ideal $B$, then $\codim Q = p$.
\end{prop}
\begin{proof}
For any associated prime $Q$ of $I_\Delta$, we have that $V(I_\Delta) \supseteq V(Q)$ and $\dim V(I_\Delta) \geq \dim V(Q)$, so $\codim Q = \abs{\vec n} - \dim V(Q) \geq \abs{\vec n} - \dim V(I_\Delta) = p$. But by Proposition 2.5(i) of \cite{BES}, an associated prime $Q$ of $I_\Delta$ that does not contain the irrelevant ideal $B$ satisfies $\codim Q \leq p$. So $\codim Q = p$.
\end{proof}

We now examine the property of gallery-connectedness in relation to the property of being virtually Cohen--Macaulay. 

A pure simplicial complex is \textit{gallery-connected} if for any two facets $F,F' \in \Delta$, there exists a path of facets $F = F_1, \dots, F_{n-1}, F_n = F'$ such that for all $1 \leq i \leq n-1$, the intersection $F_i \cap F_{i+1}$ has codimension 1 in $\Delta$. It is well-known that a Cohen--Macaulay simplicial complex is gallery-connected. Hence one would hope that a virtually Cohen--Macaulay complex is also gallery-connected, at least up to adding irrelevant faces. However, the simplicial complex in Figure~\ref{fig:2tetrahedron} provides a counterexample. This complex is on vertices $\{a,b,c,d,e,f\}$ and is contained in $\mathbb{P}^2\times\mathbb{P}^2$, with vertices $\{a,b,c\}$ in the first component and $\{d,e,f\}$ in the second component. It is virtually Cohen--Macaulay, and a short virtual resolution, with graded shifts surpressed, is below. 

\begin{equation*}\label{eqn: eg_vres} 
S^{2}\,
      \xleftarrow{\begin{pmatrix}
      0&f&0&a\\
      {-c}&{-f}&b&{-a}\end{pmatrix}}\,S^{4} \xleftarrow{\begin{pmatrix}
      0&{-b}\\
      a&0\\
      0&{-c}\\
      {-f}&0\end{pmatrix}}\,S^{2}\\
      \xleftarrow\,0
\end{equation*} 
Note that the degrees of $1$ in the two components of the first module are both $\vec{0}$. But the simplicial complex is not gallery-connected because there is no path between the facets $\{a,d,e,f\}$ and $\{b,c,d,e\}$ that satisfies the necessary conditions. Moreover, one cannot modify the complex by adding or removing irrelevant faces to obtain a gallery-connected complex since irrelevant faces are at most two-dimensional.
\begin{figure}
    \centering
    \includegraphics{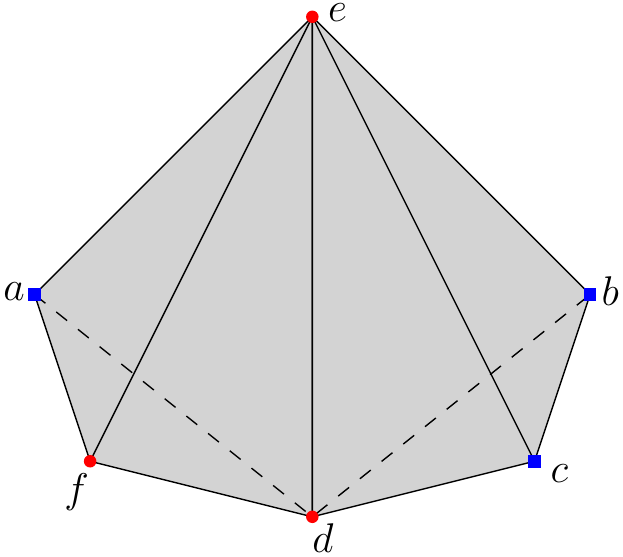}
    \caption{Two tetrahedra glued along an edge.}
    \label{fig:2tetrahedron}
\end{figure}

\section{Free resolutions as virtual resolutions}\label{sec:partial_char_vcm}

In this section, we explore a specific situation where it is sufficient to examine intersections with $I_\Delta$ by squarefree monomial ideals in order to find a short virtual resolution. We first state the main result of the section.

\begin{theorem}\label{thm:radicalVCM}
    Let $\Delta$ be a simplicial complex on the product of projective spaces $\bP^{\vec{n}} = \bP^{n_1}\times\cdots\times \bP^{n_r}$. If there exists $J$ a monomial ideal in $S$ with $V(J) = \varnothing$ such that $S/(I_\Delta\cap J)$ is Cohen--Macaulay, then there exists a simplicial complex $\Delta'$ containing only irrelevant facets such that $\sqrt{J} = I_{\Delta'}$ and  $\Delta \cup \Delta'$ is Cohen--Macaulay. Further, this implies that the complex $\Delta$ is virtually Cohen--Macaulay.
\end{theorem}

We will use the following two technical results in the proof of Theorem~\ref{thm:radicalVCM}.

\begin{lemma}\label{lem:intersectionCM}
Let $I$ be an ideal of the Cox ring $S$ of the product of projective spaces $\bP^{\vec{n}} = \bP^{n_1}\times\cdots\times \bP^{n_r}$. If there exists an ideal $J\subseteq S$ with $V(J) = \varnothing$ such that $S/(I\cap J)$ is Cohen--Macaulay, then $S/I$ is virtually Cohen--Macaulay.
\end{lemma}

\begin{proof}
Since $S/(I\cap J)$ has a free resolution of length $\codim (I\cap J)$, note that $\codim (I\cap J)\leq \codim (I)$, and by Lemma~\ref{lem::varvres}, the free resolution of $S/(I\cap J)$ is a virtual resolution of length $\codim (I)$ of $S/I$.
\end{proof}

Lemma~\ref{lem:intersectionCM} provides one way to find virtual resolutions of a given ideal. The following theorem provides another useful fact that allows us to focus on squarefree monomial ideals when checking the Cohen--Macaulay property, by passing to the radical of an arbitrary monomial ideal.

\begin{theorem}[\cite{HTT05}, Theorem 2.6]\label{thm:radicalCM}
Let $I$ be a monomial ideal of the polynomial ring $S= \Bbbk[x_1,\ldots, x_n]$ over a field $\Bbbk$. If $S/I$ is Cohen--Macaulay, then $S/\sqrt{I}$ is also Cohen--Macaulay.
\end{theorem}

\begin{proof}[Proof of Theorem~\ref{thm:radicalVCM}]
Note that for any monomial ideal $J$, $\sqrt{J}$ is generated by squarefree monomials. Therefore by the Stanley--Reisner correspondence, there exists a simplicial complex $\Delta'$ with $I_{\Delta'} = \sqrt{J}$ and $V(I_{\Delta'}) = \varnothing$. Since $I_\Delta\cap I_{\Delta'} = \sqrt{I_\Delta\cap J}$, the rest of the statement follows from Lemma~\ref{lem:intersectionCM} and Theorem~\ref{thm:radicalCM}.
\end{proof}

Note that there exist virtually Cohen--Macaulay complexes for which we cannot find a monomial ideal $J$ such that $V(J) = \varnothing$ and $I_\Delta \cap J$ is Cohen--Macaulay. In these cases it would be insufficient to limit oneself to all the possible ways one can modify the simplicial complex by irrelevant faces. Consider the following counterexample.
\begin{Counterex}\label{counterex}
Let $\Delta =\{ bdef, acef, bcdf, acdf, abdf, bcde, acde, abce \}$ be a simplicial complex in $\bP^2 \times \bP^2$ such that the vertices labeled $a,b,c$ correspond to one component of the product and those labeled $d,e,f$ correspond to the other component. In order to be virtually Cohen--Macaulay, there must exist a virtual resolution $r$ of $I_{\Delta}$ of length $2$. However, the free resolution $r$ of $S/I_{\Delta}$ has length $3$. If there was a $\Delta'$ that differs from $\Delta$ by irrelevant faces, such that the free resolution $r'$ of $S/I_{\Delta'}$ has length 2, $\Delta'$ would be at least 3-dimensional. This is because $2 = \text{length}(r') \geq \codim \left(I_\Delta'\right) = 5 - \dim \left(\Delta'\right)$. But there are no irrelevant faces that are at least three-dimensional, forcing the dimension of $\Delta'$ to be the same as the dimension of $\Delta$, so there is no such $\Delta'$. Yet, $\Delta$ is in fact found to be virtually Cohen--Macaulay by a mapping cone construction using \textit{Macaulay2} \cite{M2}. One short virtual resolution, with graded shifts surpressed, is the following:
\begin{equation*}\label{eqn: vres} 
S^{3}\,
      \xleftarrow{\begin{pmatrix}
      c\,e\,f&0&0&a\,e\,f&a\,b\,e&0&0&a\,b\,c\\
      {-c}&{-d}&b&{-a}&0&0&0&0\\
      0&0&0&0&{-e}&{-f}&d&{-c}\end{pmatrix}}\,S^{8} \xleftarrow{\begin{pmatrix}
      0&0&0&a&0\\
      0&{-b}&0&0&0\\
      a&{-d}&0&0&0\\
      b&0&0&{-c}&0\\
      {-f}&0&c&0&0\\
      e&0&0&0&{-d}\\
      0&0&0&0&{-f}\\
      0&0&{-e}&0&0\end{pmatrix}}\,S^{5}\\
      \xleftarrow\,0.
\end{equation*} 
\end{Counterex}
\section{Balanced Simplicial Complexes}\label{sec:balanced}

In this section, we consider balanced complexes on the product of projective spaces $\bP^{\vec{n}} = \bP^{n_1}\times\cdots \times \bP^{n_r}$, showing that every balanced complex is in fact virtually Cohen--Macaulay. In particular, we prove Theorem~\ref{thm:balanced_implies_vcm}.

By Lemma~\ref{lem:virtualequivalence}, it suffices to show that for every pure and balanced simplicial complex $\Delta$, there exists a simplicial complex $\Delta'$ with $V\left(I_\Delta\right)=V\left(I_{\Delta'}\right)$ such that $\Delta'$ is shellable. Recall that a \textit{shelling} of $\Delta$ is an ordered list $F_1,F_2,\ldots, F_m$ of its facets such that for all $i=2,\ldots,m$,  $\left(\bigcup_{k = 1}^{i-1} F_k\right)\cap F_i$ is pure of codimension 1. If a simplicial complex is pure and has a shelling, then it is \textit{shellable}. The following result relates shellability with the Cohen--Macaulay property.

\begin{theorem}[\cite{miller_sturmfels_2005}, Theorem 13.45]\label{thm:shellableCM}
Let $\Delta$ be a simplicial complex. If $\Delta$ is shellable, then $\Delta$ is Cohen--Macaulay.
\end{theorem}

Let $\Delta$ be a simplicial complex, and let $v \notin \Delta$ be a vertex. Then the \textit{cone} of $\Delta$ on $v$ is the simplicial complex
\[\Delta * v := \{F\cup\{v\} \mid F\in \Delta\text{ is a face}\}\cup\{\varnothing\}.\] It is clear that $\Delta$ is shellable if and only if $\Delta * v$ is shellable. 

\begin{definition}
\label{def:irrcomplex}
In the product of projective spaces $\bP^{\vec{n}}= \bP^{n_1}\times\cdots \times \bP^{n_r}$ with $r$ components, the \textit{irrelevant complex} supported on $X_{\vec{n}}$ is a pure $(r-1)$-dimensional complex $\Delta_{\text{irr}}(\vec n)$ defined as follows: for each possible $(r-1)$-dimensional face $\sigma\subseteq X_{\vec n}$,
\[\sigma \in \Delta_{\irr}(\vec{n}) \Leftrightarrow \text{there is exactly one pair of vertices $\{v, w\} \subseteq \sigma$ in the same component of $\bP^{\vec n}$}.\]
\end{definition}

The following result provides us with a shelling order of a certain complex that will help us prove our main result.

\begin{prop}\label{prop:irrshelling}
Suppose $\vec{n}$ is a length $r$ vector and has no zero entry. If $R$ is a balanced facet on $X_{\vec n}$, then $\Delta = \Delta_{\irr}(\vec{n}) \cup \{R\}$ is shellable.
\end{prop}

\begin{proof}

We will explicitly construct the shelling order on $\Delta_{\irr}(\vec{n}) \cup \{R\}$. The set of components are indexed $1,\ldots, r$, and the set of vertices $x_{i,j}$ in each component $i$ are indexed $j=0,1,\ldots,n_i$. Wihout loss of generality, let \[R = \{x_{i, 0} \mid 1 \leq i \leq r\}.\] 
For $k=1,\ldots,r$, let
\[\Delta_{\neg k} \coloneqq \{\sigma \in \Delta_{\irr}(\vec{n}) \mid x_{i,j} \in \sigma \Rightarrow i \neq k\}.\] 
In words, $R$ consists of the zeroth vertex in each component, and $\Delta_{\neg k}$ consists of faces of the irrelevant complex that do not have a vertex in component $k$. Our shelling order will first add $R$, then all facets of $\Delta_{\neg r}$, then all facets of $\Delta_{\neg (r - 1)}$, then all facets of $\Delta_{\neg (r-2)}$, and so on, until we finally add all facets of $\Delta_{\neg 1}
$. Below we describe a total order on the facets of $\Delta_{\neg k}$ that specifies how we add facets of $\Delta_{\neg k}$ in our shelling order.

Notice that a facet of $\Delta_{\neg k}$ can be specified by a pair of same-component vertices of component $i \neq k$, along with an $(r-2)$-tuple, which describes the vertices in each of the remaining $r-2$ components (all but the $i^{\text{th}}$ and $k^{\text{th}}$ components). For instance, in $\Delta_{\neg 4} \subset \Delta_{\text{irr}}((5,4,2,2))$,  the pair $(x_{2,1}, x_{2,3})$ and the $2$-tuple $(3,2)$ would specify the facet $\{x_{1,3}, x_{2,1}, x_{2,3}, x_{3,2}\}$. Given $\Delta_{\neg k}$, same-component vertices $\alpha \subset X_{\vec{n}}$, and $\vec{v} \in \Z_{\geq 0}^{r-2}$, we define $(\alpha, \vec{v})$ to be the facet of $\Delta_{\neg k}$ given by the pair $\alpha$ and the tuple $\vec{v}_F$.

Given all ordered pairs of vertices $(x_{i,j}, x_{i,\ell})$ where $i \neq k$ and  $j < \ell$, we say that $(x_{i,j}, x_{i,\ell}) < (x_{i',j'}, x_{i',\ell'})$ if either:
\begin{itemize}
    \item $i < i'$, or
    \item $i=i'$ and $(j, \ell) < (j', \ell')$ lexicographically.
\end{itemize}

With the ordering on same-component vertices above and the lexicographic ordering on the $(r - 2)$-tuples, we have a total ordering of the set of facets of $\Delta_{\neg k}$ as follows. Suppose $F_1$ and $F_2$ are facets of $\Delta_{\neg k}$, with $F_1$ corresponding to the same-component vertices $\alpha_1$ along with the $(r-2)$-tuple $\vec{v}_1$ specifying the rest of the vertices, and $F_2$ corresponding to the same-component vertices $\alpha_2$ along with the $(r-2)$-tuple $\vec{v}_2$. Then $F_1<F_2$ if and only if either:
\begin{itemize}
    \item $\vec{v}_1<\vec{v}_2$ lexicographically, or
    \item $\vec{v}_1=\vec{v}_2$ and $\alpha_1<\alpha_2$ under the total order described above.
\end{itemize}
For our shelling order, we add the facets of $\Delta_{\neg k}$ from the least to the greatest, using this order. This ordering can be extended to comparing two same-dimensional faces of $\Delta_{\neg k}$ that both contain a pair of same-component vertices and whose vertices use the same set of components. This is what we mean below when we compare two non-facets, which we represent with Greek letters such as $\Sigma, \Omega, \Gamma$.

Let $F = (\alpha, \vec{v}_F) \in \Delta_{\neg k}$. We want to show that if $F_1,\ldots,F_\ell$ are all of the facets in the shelling order before the facet $F$, then $\left(\bigcup_{i=1}^\ell F_i\right)\cap F$ is pure of codimension $1$. We consider the intersection of $F$ with the following three kinds of previous facets:

\indent\indent\textbf{Case 1:} the balanced facet $R$;\\
\indent\indent\textbf{Case 2:} a facet $G\in\Delta_{\neg m}$ such that $k<m\leq r$ (if one exists);\\
\indent\indent \textbf{Case 3:} a facet $G\in\Delta_{\neg k}$ such that $G<F$ (if one exists).

In all three cases, we show that the intersection, if nonempty, is either of codimension $1$, or contained in another intersection with another facet smaller than $F$ such that this latter intersection is of codimension $1$. It is clear that for all the facets besides the first one, their intersections with previous facets are nonempty. Thus, we ensure that $\left(\bigcup_{i=1}^\ell F_i\right)\cap F$ is pure of codimension $1$.

Some notation before we get started: given a pair of same-component vertices $\alpha \subset X_{\vec{n}}$, we denote by $\min(\alpha) \in X_{\vec{n}}$ the vertex in the pair that is indexed by a smaller number within the component. If $F = (\alpha, \vec{v}_F) \in \Delta_{\irr}(\vec{n})$, $\text{comp}(F)$ is the component of $\alpha$. We use $\delta_{x,y}$ to denote the Kronecker delta, which is 1 if $x = y$ and 0 otherwise. If $\vec{v}$ has $n$ coordinates, then for $i=1,\ldots,n$, let $\vec{v}[i]$ denote the $i^{\text{th}}$ coordinate of $\vec{v}$. Given a set of vertices $\Omega$, when $x_{i, j} \in \Omega$ is the unique vertex of $\Omega$ in component $i$, let $\Omega(i) := x_{i, j}$ and $\ind(\Omega(i)) = j$. Recall that there are three cases to consider.

\vspace{1em}
\noindent\textbf{Case 1:} First, we consider the intersection of $F$ with the balanced facet $R$. If $\vec{v}_F = (0, 0, \dots, 0)$, then either there exists $x_{i,0} \in \alpha$ and $F \cap R$ is of codimension $1$, or 
\[F \cap R \subseteq F \cap (\{x_{i,0}, \min(\alpha)\}, (0, 0, \dots, 0)),\]
where $i=\text{comp}(F)$. Notice that the latter intersection is of codimension $1$. If $\vec{v}_F \neq (0,0,\dots,0)$, then $F \cap R \subseteq R\setminus \{x_{k, 0}\}$, which is equal to the facet defined by $(\alpha, (0, 0, \dots, 0))$ intersected with $R$.
\vspace{1em}\\
\noindent\textbf{Case 2:} If $k \neq r$, then we have added facets of $\Delta_{\neg r}, \Delta_{\neg (r-1)},\ldots,\Delta_{\neg (k+1)}$ before we have added the facet $F$. Thus, we need to consider the intersection $F \cap \Delta_{\neg m}$, where $k< m\leq r$, and show that this intersection has codimension $1$. Notice that this intersection has exactly vertices in $F$ that are not in component $m$. This is pure of codimension $1$, except when $m = \text{comp}(F)$. 

Suppose that $m = \text{comp}(F)$. If $\alpha$ is not the smallest pair of same-component vertices with respect to the total ordering defined previously, then simply choose another pair of same-component vertices $\beta < \alpha$ such that $\abs{\alpha\cap \beta} = 1$. Let $G = (\beta,\vec{v}_F)$ be the facet obtained by replacing $\alpha$, the same-component vertices of $F$, with $\beta$. Then $G < F$, and $F \cap \Delta_{\neg \text{comp}(F)} \subseteq F\cap G$, which is of codimension $1$.

Now suppose that $\alpha$ is the smallest pair of same-component vertices. Then, we have that $\alpha = \{x_{\text{comp}(F), 0}, x_{\text{comp}(F), 1}\}$. If $\vec{v}_F$ has all zero entries, then $F \cap \Delta_{\neg \text{comp}(F)} \subseteq F \cap R$, which is of codimension 1. Now if $\vec{v}_F$ has some nonzero entry, then we have two cases: either $\text{comp}(F) \neq 1 + \delta_{k,1}$ or $\text{comp}(F) = 1+\delta_{k, 1}$. We will show that, in either case, we have that there exists a facet $G$ of $\Delta_{\neg k}$ such that $G<F$ and $F\cap \Delta_{\neg \text{comp}(F)}\subseteq F\cap G$, which is of codimension $1$.

Suppose we are in the first case, and $\text{comp}(F) \neq 1+\delta_{k,1}$. Then let $G = (F\setminus\{x_{\text{comp}(F),1}\})\cup\{ x_{q, \ell}\}$, where 
\[\ell=\begin{cases}
1&\quad\text{if $x_{q,0}\in F$,}\\
0&\quad\text{otherwise,}\end{cases}\qquad\qquad q=\begin{cases}
\text{comp}(F)-2&\quad\text{if $k=\text{comp}(F)-1$,}\\
\text{comp}(F)-1&\quad\text{otherwise.}
\end{cases}\]
Then $G$ is a facet defined by same-component vertices $\beta$ and vector $\vec{v}_G$ such that $\text{comp}(G) = q$. Notice that $G < F$ because $\vec{v}_G \leq \vec{v}_F$ and $\beta < \alpha$ by construction. 

Now suppose we are in the second case, and $\text{comp}(F) = 1+\delta_{k,1}$. Let $G = (F \setminus \{x_{\text{comp}(F), 1}\}) \cup\{ x_{r, \ell}\}$, where $\ell = 1$ if $x_{r, 0} \in F$, and $\ell = 0$ otherwise. Again, let $G$ be defined by the same-component vertices $\beta$ and the vector $\vec{v}_G$. Then $G < F$ because $\vec{v}_G < \vec{v}_F$. To see this, notice that $\vec{v}_G$ is the result of appending $0$ to the front of $\vec{v}_F$ and removing the end of $\vec{v}_F$. This is always less than or equal to $\vec{v}_F$, with equality obtained only if $\vec{v}_F$ is all zero, which is not the case. We may assume $\vec{v}_F[1] = 0$; otherwise $\vec{v}_G < \vec{v}_F$ automatically. Let $w$ be the smallest index such that $\vec{v}_F[w + 1] > \vec{v}_F[w]$, which exists because $\vec{v}_F$ has nonzero entries. Then, $\vec{v}_F[i] = 0$ for all $i \leq w$. Note that $\vec{v}_G[i+1] = \vec{v}_F[i] = \vec{v}_F[i+1]$ for $i < w$, and $\vec{v}_G[w + 1] = \vec{v}_F[w] < \vec{v}_F[w+1]$. Thus, we have that $\vec{v}_G<\vec{v}_F$ lexicographically, so we have that $G<F$. 

Note that in both cases, we have found a facet $G$ with $G < F$ and $F \cap \Delta_{\neg m} \subseteq F \cap G$, which is of codimension $1$. 

\vspace{1em}
\noindent\textbf{Case 3:} We now show that for $G \in \Delta_{\neg k}$ such that $G < F$, we have that $F \cap G$ is either empty, of codimension $1$, or contained in an intersection $F \cap G'$ which is of codimension $1$ and such that $G'\in\Delta_{\neg k}$ is a facet with $G' < F$. 

Suppose that $F \cap G$ is nonempty. Let $F$ and $G$ be the facets defined by $(\alpha,\vec{v}_F)$ and $(\beta,\vec{v}_G)$ respectively, where $\alpha$ and $\beta$ are pairs of same-component vertices, and $\vec{v}_F$ and $\vec{v}_G$ are $(r-2)$-tuples. We proceed by casework. The two main cases that we consider are if $\text{comp}(F)=\text{comp}(G)$ and if $\text{comp}(F)\neq\text{comp}(G)$. 

First suppose that $\text{comp}(F) = \text{comp}(G)$. Then one can find an appropriate facet $G'$, depending on whether $\vec{v}_G < \vec{v}_F$ or $\vec{v}_G=\vec{v}_F$ under the lexicographic ordering (note that $G<F$ by assumption, so these are the only two cases to consider).
If $\vec{v}_G < \vec{v}_F$, choose $\vec{w} < \vec{v}_F$ to be an appropriate $(r-2)$-tuple such that $\vec{w}$ and $\vec{v}_F$ are the same in $r-1$ coordinates, and the coordinates in which $\vec{v}_F$ and $\vec{v}_G$ are the same as the subset of these $r-3$ coordinates.

Then, let $G'$ be the facet defined by $(\alpha, \vec{w})$; by the way that it is defined, $G'<F$ and $G'\cap F$ is of codimension $1$. Now suppose $\vec{v}_G = \vec{v}_F$; then, $\beta<\alpha$. This implies that either $\alpha\cap\beta\neq\varnothing$ (when considering $\alpha$ and $\beta$ as sets) and $F\cap G$ is of codimension $1$, or $\alpha\cap\beta=\varnothing$. In this case, there exists a pair of same-component vertices $\gamma$ such that $\gamma<\alpha$ and $\abs{\gamma\cap\alpha}=1$, since $\alpha$ is not the smallest in the total order. Let $G'$ be the facet defined by $(\gamma,\vec{v}_F)$. Again, we have by definition that $G'<F$ and $G'\cap F$ is of codimension $1$.

Next we consider the case that $\text{comp}(F)\neq\text{comp}(G)$. Let $\Sigma_F$ and $\Sigma_G$ be the sets of vertices in $F$ and $G$ respectively that are in components that have indices strictly less than $\min(\text{comp}(F), \text{comp}(G))$. Let $\Omega_F$ and $\Omega_G$ be the sets of vertices in $F$ and $G$ that are in components that have indices between and including $\text{comp}(F)$ and $\text{comp}(G)$. Finally, let $\Gamma_F$ and $\Gamma_G$ be $F\setminus(\Sigma_F\cup\Omega_F)$ and $G\setminus(\Sigma_G \cup\Omega_G)$. To clarify, $F=\Sigma_F\cup\Omega_F\cup\Gamma_F$, where $\Sigma_F$ consists of the vertices with the smallest indices and $\Gamma_F$ consists of the vertices with the largest indices.

By abuse of notation, we also consider the sets of vertices $\Sigma_i,\Omega_i,$ and $\Gamma_i$ as tuples, ordered by the indices of the vertices. By our ordering of facets, we have that $\Sigma_G\leq \Sigma_F$ lexicographically since $G<F$. Recall that we are currently in the case that $\text{comp}(F)\neq\text{comp}(G)$. Within this case, we consider two subcases: $\Sigma_G<\Sigma_F$ and $\Sigma_G=\Sigma_F$. If $\Sigma_G<\Sigma_F$, we will show that there exists a facet $G'$ such that $G'<F$ and $G'\cap F$ is of codimension $1$. If $\Sigma_G=\Sigma_F$, then we must consider further subcases, depending on the difference between $\text{comp}(F)$ and $\text{comp}(G)$.

We first consider the subcase that $\Sigma_G<\Sigma_F$. In this case, choose the smallest integer $i$ such that a vertex in $\Sigma_G$ and in component $i$ has a smaller index than a vertex in $\Sigma_F$ and in component $i$; let $\Sigma_G(i)$ and $\Sigma_F(i)$ refer to these vertices. Then, let $G'$ be the facet $G'= (F\setminus \Sigma_F(i))\cup\Sigma_G(i)$. By definition, we have that $G' < F$ and $G \cap F \subseteq G' \cap F$, which is pure of codimension $1$.

Next, we consider the second subcase where $\Sigma_F=\Sigma_G$. We now must restrict our attention to $\Omega_F$ and $\Omega_G$. $\Omega_F$  consists of a pair of same-component vertices in component $\text{comp}(F)$ and other vertices in distinct components. So we may write $\Omega_F$ as $(\alpha,\vec{u}_F)$, where $\vec{u}_F$ specifies the indices of the vertices of $\Omega_F$ in the components other than $\text{comp}(F)$. Similarly, we may write $\Omega_G$ as $(\beta,\vec{u}_G)$. 

We want to show that there exists a particular set of vertices, which we call $\Omega_{G'}$, such that $\Omega_{G'}<\Omega_F$ and $\Omega_F \cap \Omega_G\subseteq\Omega_F\cap\Omega_G'$, with this second intersection being of codimension $1$ in $\Omega_F$. We will then form a facet $G'\in \Delta_{\neg k}$ later based on $\Omega_{G'}$ to satisfy the necessary conditions; namely, that $G'< F$ and $F\cap G\subseteq F\cap G'$, with this latter intersection being of codimension $1$.
We consider the following subcases: 
\begin{itemize}
    \item $\text{comp}(G)-\text{comp}(F)=1$,
    \item $\text{comp}(G)-\text{comp}(F)>1$,
    \item $\text{comp}(F)-\text{comp}(G)=1$, or
    \item $\text{comp}(F)-\text{comp}(G)>1$.
\end{itemize}

First consider the case where $\text{comp}(G)-\text{comp}(F)=1$. We then have the following possibilities:
\begin{enumerate}
\item $\Omega_G(\text{comp}(F)) \in \alpha$, in which case $\Omega_F \cap \Omega_G$ is already codimension $1$ in $\Omega_F$;
\item $\Omega_G(\text{comp}(F)) \notin \alpha$ and $\min(\alpha) < \Omega_F(\text{comp}(G))$ (comparing only the second subscripts of the two vertices, i.e., their indices in their respective components), in which case we can let $\Omega_{G'} = (\Omega_G\setminus\Omega_G(\text{comp}(F)))\cup\min(\alpha)$ and maintain $\Omega_{G'}<\Omega_F$;
\item $\Omega_G(\text{comp}(F)) \notin \alpha$ and $\alpha$ is not the smallest pair of same-component vertices, in which case we can let $\Omega_{G'} = \gamma \cup\Omega_F(\text{comp}(G))$, where $\gamma$ is a pair of vertices in component $\text{comp}(F)$ such that $\gamma<\alpha$ and $\abs{\gamma\cap\alpha} = 1$; note that $\min(\alpha) > \Omega_F(\text{comp}(G))$ falls in this case.
\end{enumerate}
These possibilities are exhaustive. Indeed, if $\Omega_G(\text{comp}(F)) \notin \alpha$, $\alpha$ is the smallest pair of same-component vertices, and $\min(\alpha) = \Omega_F(\text{comp}(G))$, then $\Omega_F(\text{comp}(G)) = x_{\text{comp}(G), 0}$. Note that $\vec{v}_G < \vec{v}_F$, so we have $\Omega_G(\text{comp}(F)) \leq \Omega_F(\text{comp}(G)) = x_{\text{comp}(G), 0}$, so $\Omega_G(\text{comp}(F)) = x_{\text{comp}(G), 0}$. But then $\Omega_G(\text{comp}(F)) \in \alpha$, which is a contradiction. So this case cannot happen.

We now consider the case where $\text{comp}(G)-\text{comp}(F)>1$. Recall that we are assuming that $\Sigma_F=\Sigma_G$, so by the lexicographic ordering, we must have that $\ind(\Omega_G(\text{comp}(F))) \leq \ind(\Omega_F(\text{comp}(F)+i))$, where $i = 1 + \delta_{\text{comp}(F) + 1, k}$ (recall that we are in working in $\Delta_{\neg k}$). 
If $\ind(\Omega_G(\text{comp}(F))) < \ind(\Omega_F(\text{comp}(F) + i))$, then let 
\[\Omega_{G'}= (\Omega_F\setminus(\alpha\cup\Omega_F(\text{comp}(G)))\cup(\Omega_G(\text{comp}(F))\cup\beta).\] 
This is essentially $\Omega_G$  with vertices of the intermediate components replaced with the corresponding vertices in $\Omega_F$. Since $\Omega_G(\text{comp}(F)) < \Omega_F(\text{comp}(F)+i)$, we have that $\vec{u}_{G'} < \vec{u}_F$. As desired, $\Omega_G \cap \Omega_F \subseteq \Omega_{G'} \cap \Omega_F$, which is of codimension $1$.

If, on the other hand, we have that $\ind(\Omega_G(\text{comp}(F))) = \ind(\Omega_F(\text{comp}(F)+i))$, we consider the following three cases that depend on the relationship between $\ind(\Omega_F(\text{comp}(F)+i))$ and $\ind(\Omega_G(\text{comp}(F)+i))$. 
\begin{enumerate}
\item If $\ind(\Omega_F(\text{comp}(F)+i)) < \ind(\Omega_G(\text{comp}(F) + i))$, then let 
\[\Omega_{G'} = \Omega_F\setminus(\alpha \cup \Omega_F(\text{comp}(G)))\cup (\Omega_G(\text{comp}(F))\cup \beta).\] 
Then $\vec{u}_{G'} < \vec{u}_G$, because we replaced $\Omega_G(\text{comp}(F) + i)$ with the smaller $\Omega_F(\text{comp}(F)+i)$. Hence $\Omega_{G'} < \Omega_G < \Omega_F$ and $\Omega_G \cap \Omega_F \subseteq \Omega_{G'} \cap \Omega_F$, which has codimension $1$.
\item If $\ind(\Omega_F(\text{comp}(F)+i)) = \ind(\Omega_G(\text{comp}(F) + i))$, then $\codim(\Omega_F \cap \Omega_G) = \codim(\Omega_{F'} \cap \Omega_{G'})$ where $\Omega_{F'}$ is the face $\Omega_{F'} = \Omega_F\setminus \Omega_F(\text{comp}(F) + i)$ and $\Omega_{G'}$ is the face $\Omega_{G'} = \Omega_G\setminus \Omega_G(\text{comp}(F) + i)$. We are allowed to pass to a case where $\Omega_F$ and $\Omega_G$ have fewer vertices because we have treated the case where $\text{comp}(F) + 1 = \text{comp}(G)$.
\item If $\ind(\Omega_F(\text{comp}(F)+i)) > \ind(\Omega_G(\text{comp}(F) + i))$, then $\Omega_F(\text{comp}(F) + i)$ is not in $\Omega_F\cap\Omega_G$ and we can let 
\[\Omega_{G'} = (\Omega_G \setminus \Omega_F(\text{comp}(F) + i)) \cup \Omega_G(\text{comp}(F) + i).\] We then see that $\Omega_{G'}$ satisfies the necessary conditions; in particular, $\Omega_{G'}<\Omega_F$.
\end{enumerate}

Thus far, we have shown in the cases where $\text{comp}(G)-\text{comp}(F)=1$ and $\text{comp}(G)-\text{comp}(F)>1$ that we can find $\Omega_{G'}$ satisfying the necessary conditions. 

We now consider the two cases where $\text{comp}(F) > \text{comp}(G)$. Since $\vec{u}_G \leq \vec{u}_F$, we are free to alter $\beta$ such that $\Omega_F(\text{comp}(G)) \in \beta$ as before. 

First, we consider the case where $\text{comp}(F)-\text{comp}(G)=1$. We have three cases: 

\begin{enumerate}
\item $\Omega_G(\text{comp}(F)) \in \alpha$, in which case $\Omega_F \cap \Omega_G$ is already of codimension $1$ in $\Omega_F$;
\item $\Omega_G(\text{comp}(F)) \notin \alpha$ and $\min(\alpha) \leq \Omega_F(\text{comp}(G))$, in which case let $\Omega_{G'} = \beta \cup \min(\alpha) $. We then have that $\Omega_{G'} < \Omega_F$ because $\min(\alpha) \leq \Omega_F(\text{comp}(G))$ and $\text{comp}(G) < \text{comp}(F)$;
\item $\Omega_G(\text{comp}(F)) \notin \alpha$ and $\min(\alpha) > \Omega_F(\text{comp}(G))$, in which case $\alpha$ is not the smallest pair of same-component vertices, and we can let $\Omega_{G'} = \gamma \cup \Omega_F(\text{comp}(G))$, where $\gamma$ is a pair of vertices in component $\text{comp}(F)$ such that $\gamma < \alpha$ and $\abs{\gamma\cap\alpha} = 1$.
\end{enumerate}

Now consider the case where $\text{comp}(F)-\text{comp}(G)>1$. Again, we consider three cases depending on the relationship between $\Omega_F(\text{comp}(G))$, $\Omega_F(\text{comp}(G)+i)$, and $\Omega_G(\text{comp}(G)+i)$, where $i = 1 + \delta_{\text{comp}(G)+1, k}$.
\begin{enumerate}
\item If $\ind(\Omega_F(\text{comp}(G))) > \ind(\Omega_F(\text{comp}(G)+i))$, then let 
\[\Omega_{G'} = (\Omega_F\setminus( \Omega_F(\text{comp}(G)) \cup\alpha)) \cup (\beta \cup \Omega_G(\text{comp}(F))).\] 
Then, the first vertex after $\beta$ in $\Omega_{G'}$ is $\Omega_F(\text{comp}(G)+i)$, and so $\Omega_{G'} < \Omega_F$.
\item If $\ind(\Omega_F(\text{comp}(G))) = \ind(\Omega_F(\text{comp}(G)+i)) = \ind(\Omega_G(\text{comp}(G)+i))$, then we have $\codim(\Omega_F \cap \Omega_F) = \codim(\Omega_{F'} \cap \Omega_{G'})$, where $\Omega_{F'}$ and $\Omega_{G'}$ are the faces $\Omega_{F'} = \Omega_F \setminus \Omega_F(\text{comp}(G)+i)$ and $\Omega_{G'} = \Omega_G \setminus\Omega_G(\text{comp}(G)+i)$. We are allowed to pass to a case where $\Omega_F$ and $\Omega_G$ have fewer vertices because we have already considered the case where $\text{comp}(G) + 1 = \text{comp}(F)$.
\item If either $\ind(\Omega_F(\text{comp}(G)) < \ind(\Omega_F(\text{comp}(G)+i))$ or \[\ind(\Omega_F(\text{comp}(G)))=\ind(\Omega_f(\text{comp}(G)+i))>\ind(\Omega_G(\text{comp}(G)+i)),\]  then let $\Omega_{G'} = (\Omega_F\setminus\Omega_F(\text{comp}(G)+1))\cup\Omega_G(\text{comp}(G)+i) $. 
In either situation, $\Omega_F(\text{comp}(G)+i) \notin \Omega_F\cap\Omega_G$: in the first case, \[\ind(\Omega_G(\text{comp}(G)+i)) \leq \ind(\Omega_F(\text{comp}(G))) < \ind(\Omega_F(\text{comp}(G)+i));\] while in the second case, we are given that $\ind(\Omega_G(\text{comp}(G)+i))<\ind(\Omega_F(\text{comp}(G)+i))$. In both cases, we have that $\Omega_{G'}<\Omega_F$ because we replaced $\Omega_F(\text{comp}(G)+i)$ with the smaller $\Omega_G(\text{comp}(G)+i)$.
\end{enumerate}

In all cases of comparing $\text{comp}(F)$ and $\text{comp}(G)$, we have that there exists $\Omega_{G'}$ such that $\Omega_{G'}<\Omega_F$ and $\Omega_F\cap\Omega_G\subseteq\Omega_F\cap\Omega_{G'}$, which is of codimension $1$ in $\Omega_F$. Now define the facet $G'\in\Delta_{\neg k}$ to be $G'=\Sigma_F\cup\Omega_{G'}\cup\Gamma_F$. Then $G'<F$, and $F\cap G\subseteq F\cap G'$, which is of codimension $1$.

We have therefore shown that, if $F_1,\ldots,F_\ell$ are the facets in the ordering before $F$, then $\left(\bigcup_{i=1}^\ell F_i\right)\cap F$ is pure of codimension $1$. In conclusion, our ordering is indeed a shelling order.
\end{proof}

We are now ready to show Theorem~\ref{thm:balanced_implies_vcm}, the main result of this section.

\begin{proof}[Proof of Theorem~\ref{thm:balanced_implies_vcm}]
If $\vec{n}$ has all zero entries, then $\Delta$ has exactly one balanced facet and is trivially shellable. For $\vec{n} \neq \vec{0}$, we claim that there exists some set of irrelevant facets $\Delta'$ such that $\Delta' \cup \Delta$ is shellable, and therefore $\Delta$ is virtually Cohen-Macaulay. We prove this by induction on the number of zero entries in $\vec{n}$.

If $\vec{n}$ has no zero entries, then $\Delta_{\irr}(\vec{n}) \cup R$ has a shelling order given by Proposition~\ref{prop:irrshelling}, where $R$ is some balanced facet in $\Delta$. Then add the remaining balanced facets of $\Delta$ in any order. The ordering is still a shelling because given any balanced facet, all of its ridges are contained in $\Delta_{\irr}(\vec{n})$.

Note that if $\vec{m}$ is the result of permuting entries of $\vec{n}$, then any shelling order of a complex on the vertex set $X_{\vec m}$ gives rise to a shelling order of the relabelled complex on $X_{\vec n}$. Therefore, if $\vec{n}$ has some zero entries but also some nonzero entries, we may assume that all zero entries in $\vec{n}$ are trailing zeros. Suppose that $\vec{m}$ is $\vec{n}$ with a zero appended. By the induction hypothesis, we have a shelling order of $\Delta' \cup \Delta|_{X_{\vec n}}$, where $\Delta'$ is some complex of irrelevant facets. Note that any balanced facet on $X_{\vec m}$ must contain the vertex $x_{r, 0}$ of the last component, since this is the only vertex in this component. Therefore $\Delta = \Delta|_{X_{\vec n}} * x_{r,0}$ and the previous shelling order gives rise to a shelling order of $(\Delta' \cup \Delta|_{X_{\vec n}})*x_{r,0} = \Delta' * x_{r,0} \cup \Delta$, where $\Delta' * x_{r,0}$ consists of only irrelevant facets.
\end{proof}

\section{acknowledgements}
This research was conducted at the 2019 University of Minnesota--Twin Cities REU, supported by NSF RTG grant DMS-1745638. We thank Christine Berkesch and Jorin Schug for their advice and support. We also thank Gregory Michel for editing this paper, and we thank Vic Reiner for his help with \secref{sec:balanced}.

\end{document}